\documentclass{amsart}
	\usepackage[latin1]{inputenc}
	\usepackage{ae,aecompl,amsbsy,amssymb,amsmath,amsthm,
eurosym,amsfonts,epsfig,graphicx,graphics,verbatim,enumerate,esint,color,MnSymbol}
	
\usepackage{color,soul}

\definecolor{lightblue}{rgb}{.85,.93,1}
\sethlcolor{lightblue}
\newcommand{\cd}{\mathcal{D}}
\newcommand{\cs}{\mathcal{S}}
\newcommand{\be}{\begin{equation*}}
\newcommand{\ee}{\end{equation*}}
\newcommand{\ba}{\begin{eqnarray*}}
\newcommand{\ea}{\end{eqnarray*}}

\newcommand{\rn}{\mathbb{R}^d}

\newcommand{\rone}{\mathbb{R}}
\newcommand{\oh}{\frac{1}{2}}
\newcommand{\Q}{|Q|}

\newcommand{\normbee}[1]{\lVert#1\rVert_{B(E\to E)}}
\newcommand{\norm}[1]{\lVert#1\rVert_{E}}
\newcommand{\normnone}[1]{\lVert#1\rVert}
\newcommand{\te}{T}

\newcommand{\normlpwrnr}[1]{\lVert#1\rVert_{L_w^p(\rn\to\rone)}}

\newcommand{\normlpwv}[1]{\lVert#1\rVert_{L_w^p(\rn\to E)}}
\newcommand{\normlpwrne}[1]{\lVert#1\rVert_{L_w^p(\rn\to E)}}
\newcommand{\normltwowrne}[1]{\lVert#1\rVert_{L_w^2(\rn\to E)}}

\newcommand{\normlpsrnr}[1]{\lVert#1\rVert_{L_\sigma^p(\rn\to\rone)}}
\newcommand{\normlpsrne}[1]{\lVert#1\rVert_{L_\sigma^p(\rn\to E)}}

\newcommand{\normlpsrnlq}[1]{\lVert#1\rVert_{L_\sigma^p(\rn\to \ell^q)}}

\newcommand{\normlprne}[1]{\lVert#1\rVert_{L^p(\rn\to E)}}
\newcommand{\normlpwrnlq}[1]{\lVert#1\rVert_{L^p_w(\rn\to \ell^q)}}

\newcommand{\normlonerne}[1]{\lVert#1\rVert_{L^1(\rn\to E)}}

\newcommand{\normweaklprne}[1]{\lVert#1\rVert_{L^{p,\infty}(\rn\to E)}}
\newcommand{\normweaklonerne}[1]{\lVert#1\rVert_{L^{1,\infty}(\rn\to E)}}
\newcommand{\normltwowr}[1]{\lVert#1\rVert_{L_w^2(\rn\to\rone)}}

\newcommand{\abs}[1]{\lvert#1\rvert}
\newcommand{\measure}[1]{\lvert#1\rvert}

\newcommand{\charatwo}{[w]_{A_2}}
\newcommand{\charap}{[w]_{A_p}}

\newcommand{\lprnr}{L^p(\rn\to \rone)}
\newcommand{\lonernr}{L^1(\rn\to \rone)}

\newcommand{\lprne}{L^p(\rn\to E)}

\newcommand{\lpwrnr}{L^p_w(\rn\to \rone)}
\newcommand{\ltwowrnr}{L^2_w(\rn\to \rone)}

\newcommand{\lpsrne}{L^p_\sigma(\rn\to E)}

\newcommand{\lpsrnlq}{L^p_\sigma(\rn\to \ell^q)}

\newcommand{\lpwrne}{L^p_w(\rn\to E)}

\newcommand{\ltwowrne}{L^2_w(\rn\to E)}

\newcommand{\m}{c_\lambda(f;Q)}

\theoremstyle{plain}
\newtheorem{thm}{Theorem}[section]
\newtheorem{lemma}[thm]{Lemma}
\newtheorem{corollary}[thm]{Corollary}

\theoremstyle{definition}
\newtheorem{defn}[thm]{Definition}
\newtheorem*{notation}{Notation}

\newtheorem{lem}[thm]{Lemma}

\newcommand{\thesymbolforalpha}{\rho}
\theoremstyle{definition}

\theoremstyle{remark}
\newtheorem*{rem}{Remark}

\numberwithin{equation}{section}

\begin{document}

\title{The $A_2$ theorem and the local oscillation decomposition for Banach space valued functions}
\author{Timo S. H\"anninen}
\address{Department of Mathematics and Statistics, University of Helsinki, P.O. Box 68, FI-00014 HELSINKI, FINLAND}
\email{timo.s.hanninen@helsinki.fi}
\author{Tuomas P. Hyt\"onen}
\address{Department of Mathematics and Statistics, University of Helsinki, P.O. Box 68, FI-00014 HELSINKI, FINLAND}
\email{tuomas.hytonen@helsinki.fi}

\date{\today}

\thanks{Both authors are supported by the European Union through the ERC Starting Grant "Analytic-probabilistic methods for borderline singular integrals". T. P. H. is also supported by the Academy of Finland, grants 130166 and 133264.}
     

\begin{abstract}
We prove that the operator norm of every Banach space valued Calder\'{o}n-Zygmund operator $T$ on the weighted Lebesgue-Bochner space depends linearly on the Muckenhoupt $A_2$ characteristic of the weight. In parallel with the proof of the real-valued case, the proof is based on pointwise dominating every Banach space valued Calder\'{o}n-Zygmund operator by a series of positive dyadic shifts. In common with the real-valued case, the pointwise dyadic domination relies on Lerner's local oscillation decomposition formula, which we extend from the real-valued case to the Banach space valued case. The extension of Lerner's local oscillation decomposition formula is based on a Banach space valued generalization of the notion of median.
\end{abstract}
\maketitle
\section{Introduction}
In this paper we introduce a Banach space valued generalization of a median. Using the generalization of a median, we extend Lerner's local oscillation decomposition formula from real-valued functions to Banach space valued functions. As an application of the extension of Lerner's local oscillation decomposition formula, we prove that, in common with the real-valued case, every Banach space valued Calder\'{o}n-Zygmund operator is pointwise dominated by a series of positive dyadic shifts. As an immediate consequence of the pointwise dyadic domination, we obtain the $A_2$ theorem for Banach space valued Calder\'{o}n-Zygmund operators.

Let $A_2$ denote the class of all the weights with a finite Muckenhoupt $A_2$ characteristic $[w]_{A_2}$. The $A_2$ theorem states that for each Calder\'{o}n-Zygmund  operator $T$ with the H\"older exponent $\alpha\in(0,1]$ we have
\begin{equation*}
\label{a2theorem}
\normltwowr{Tf}\leq C_T\, \charatwo \normltwowr{f}\quad\text{for all $f\in\ltwowrnr$ and for all $w\in A_2$}.
\end{equation*}
The $A_2$ theorem in full generality was first proven  by Hyt\"{o}nen \cite{arXiv:1007.4330}. The result was preceded by many intermediate results by others. See \cite{arXiv:1202.2824} for a list of contributions to the $A_2$ theorem. The proof in \cite{arXiv:1007.4330} consists of two steps: The first step is to pointwise represent every Calder\'{o}n-Zygmund operator as a series (over complexity $k$) of dyadic shift operators (with complexity $k$) averaged over an infinite number of randomized dyadic grids. The second step is to obtain the $A_2$ estimate for the dyadic shift operators (with such a decay in complexity $k$ that the series converges). 

Hyt\"{o}nen, Lacey, and P\'{e}rez \cite{arXiv:1202.2229}, and Lerner \cite{arXiv:1202.2824} showed that every Calder\'{o}n-Zygmund operator is pointwise dominated by a series (over complexity $k$) of simple positive dyadic shift operators $A_{\cs,k}$ (with complexity $k$) summed over a finite number of translated dyadic grids (parameterized by $u\in\{0,\frac{1}{3},\frac{2}{3}\}^d$),
\begin{equation}
\label{pointwisedominationreal}
\abs{(Tf)(x)}\leq C_T \sum_{u\in\{0,\frac{1}{3},\frac{2}{3}\}^d}\sum_{k=0}^\infty 2^{-\alpha k}(A_{\cs^u_k,k} \abs{f})(x) \quad\text{for a.e. $x\in\rn$}.
\end{equation}
This result simplifies the first proof of the $A_2$ theorem \cite{arXiv:1007.4330}, because the pointwise domination \eqref{pointwisedominationreal} is simpler than the representation theorem in \cite{arXiv:1007.4330} and because the $A_2$ estimate is obtained more simply for the operator $A_{\cs,k}$ than for a general dyadic shift operator.

Moreover, Lerner \cite{arXiv:1202.1860} proved that the formal adjoint $A_{\cs,k}^\star$ of each operator $A_{\cs,k}$ is pointwise dominated (linearly in complexity $k$) by the operator $A_{\cs,k=0}$. Hence, by duality and the self-adjointness of $A_{\cs,k=0}$, the $A_2$ estimate for the operator $A_{\cs,k}$ follows from the $A_2$ estimate for the operator $A_{\cs,k=0}$, as shown in \cite{arXiv:1202.1860}. This result simplifies further the proof of the $A_2$ theorem, because the $A_2$ estimate for the operator $A_{\cs,k=0}$ is simple to obtain, as shown in \cite[The proof of Theorem 1]{cruz-uribe2010}. See \cite{arXiv:1202.2824} for a self-contained proof of the $A_2$ theorem based on the simplifications mentioned. Both of the results on domination \cite{arXiv:1202.2229,arXiv:1202.2824} and \cite{arXiv:1202.1860} are based on Lerner's local oscillation decomposition formula \cite{lerner2010,arXiv:1202.1860}.

In our paper we extend the results discussed in the preceding paragraphs from real-valued functions to Banach space valued functions.
In what follows we summarize the results in an informal manner.  The results together with the definitions are stated formally in Section 2. Let $(E,\norm{\,\cdot\,})$ be a Banach space. Suppose that $T$ is an $E$-valued Calder\'{o}n-Zygmund operator with the H\"older exponent $\alpha\in(0,1]$ on the Lebesgue-Bochner space $\lprne$. Assume that $f:\rn \to E$ is a Bochner measurable function.

In this paper we prove that, in common with the real-valued case, for each $E$-valued Calder\'{o}n-Zygmund operator we have the pointwise dyadic domination theorem
\begin{equation*}
\label{pointwisedominationvector}
\norm{(Tf)(x)}\leq C_T \sum_{u\in\{0,\frac{1}{3},\frac{2}{3}\}^d}\sum_{k=0}^\infty 2^{-\alpha k}(A_{\cs^u_k,k} \norm{f})(x)  \quad\text{for a.e. $x\in\rn$},
\end{equation*}
and, as a corollary, the $A_2$ theorem
$$
\normltwowrne{Tf}\leq C_T\,  [w]_{A_2} \normltwowrne{f}\quad\text{for all $f\in\ltwowrne$ and for all $w\in A_2$}.
$$

Once we have an $E$-valued generalization of Lerner's local oscillation decomposition formula, the proof of the $E$-valued dyadic domination theorem proceeds in parallel with the proof of the real-valued dyadic domination theorem. The difficulty in extending Lerner's formula from real-valued functions to $E$-valued functions is that the formula is derived using the notion of a median, notion which is based on the ordering of the real line. We circumvent the difficulty by introducing an $E$-valued generalization of a median, which we call {\it a quasi-optimal center of oscillation} and denote by $c_\lambda(f,Q)$. By using the notion of a quasi-optimal center of oscillation, we extend Lerner's local oscillation formula to $E$-valued functions,
$$
1_{Q^0}(x)\norm{f(x)-c_{1/4}(f;Q^0)}\leq 12\sum_{Q\in\cs }\omega_{2^{-d-3}}(f;Q)1_{Q}(x)\quad\text{for a.e. $x\in\rn$}.
$$ 

We note that two-weight norm inequalities of the form
\begin{equation*}\normlpwrnlq{Nf}\leq C_{N,w,\sigma}\normlpsrnlq{f}\quad \text{for all $f\in\lpsrnlq$}
\end{equation*}
were studied in \cite[Section 8]{cruz-uribe2012} for a $\ell^q$-valued maximal operator and in \cite{scurry2010} for another $\ell^q$-valued operator. The setting in \cite[Section 8]{cruz-uribe2012} and \cite{scurry2010} differs from ours, because neither the operator studied in \cite[Section 8]{cruz-uribe2012}  nor the operator studied in \cite{scurry2010} is (albeit each one is similar to) a $\ell^q$-valued Calder\'{o}n-Zygmund operator and because instead of $\ell^q$-valued functions we study $E$-valued functions for an abstract Banach space $E$.

Our paper is organized as follows. In Section 2 we first introduce the setting along with the notation. Then we state the pointwise dyadic domination theorem, Theorem \ref{introductiondyadicdomination}, and the $A_2$ theorem, Corollary \ref{vvapt},  for Banach space valued Calder\'{o}n-Zygmund operators. We conclude Section 2 by defining the notion of a quasi-optimal center of oscillation and by stating the Banach space valued generalization of Lerner's local oscillation decomposition formula, Theorem \ref{lernersformula}. In Section 3 we prove, assuming the generalization of Lerner's formula, the pointwise dyadic domination theorem and the $A_2$ theorem for Banach space valued Calder\'{o}n-Zygmund operators. In Section 4 we prove the generalization of Lerner's formula.

\section{Vector-valued setting and the main theorems}
The material from Definition \ref{bochnermeasurability} to Definition \ref{vvczo} consists of defining the notions of a vector-valued Lebesgue-Bochner space, of a vector-valued Calder\'{o}n-Zygmund operator, and of a Muckenhoupt weight. The reader familiar with these notions may prefer to move on to Definition \ref{nearlydisjoint}.

\begin{notation} Let $(E,\norm{\,\cdot\,})$ be a Banach space. Denote by $B(E\to E)$ the space of bounded linear operators from $E$ to $E$, and denote by $\normbee{\,\cdot\,}$ the usual operator norm. Let $(\rn,\mathcal{L}(\rn),\measure{\,\cdot\,})$ denote the Lebesgue measure space. Denote by $B_E(c,r)$ the closed ball  with center $c\in E$ and radius $r>0$ in $E$. Let $\norm{f}$ denote the function $\rn\ni x\rightmapsto \norm{f(x)}\in[0,\infty)$.

Suppose that $A$ and $B$ are sets. Let $g:A\times B\to[0,\infty]$ and $h:A\times B\to[0,\infty]$ be functions. The notation "$g(a,b)\lesssim_b h(a,b)$" and the notation "$g(a,b)\lesssim h(a,b)$ for all $a\in A$" both mean that for each $b\in B$ there exists a constant $C_b>0$ such that $g(a,b)\leq C_b\; h(a,b)$ for all $a\in A$.
\end{notation}

\begin{defn}[Bochner measurability]\label{bochnermeasurability} 

A function $f:\rn\to E$ is called {\it (Lebesgue) measurable}, if and only if $f^{-1}(B)\in\mathcal{L}(\rn)$ for every Borel set $B$ of $E$. 

A function $f:\rn\to E$ is called {\it essentially separably valued (with respect to the Lebesgue measure space)}, if and only if there exist a set $N\in\mathcal{L}(\rn)$ of measure zero $\measure{N}=0$ such that the image $f(\rn\setminus N)$ of the complement $\rn\setminus N$ of $N$ is separable. 

A function $f:\rn\to E$ is called {\it strongly measurable (with respect to the Lebesgue measure space)} or {\it Bochner measurable (with respect to the Lebesgue measure space)}, if and only if it is both essentially separably valued and Lebesgue measurable. 

\end{defn}
\begin{defn}[Weight function and weight measure]
A locally integrable function $w:\rn\to(0,\infty)$ is called {\it a weight function}. A weight function $w:\rn\to(0,\infty)$  gives rise to {\it the weight measure} $w:\mathcal{L}(\rn)\to[0,\infty]$ by setting
$$
w(A):=\int_{A} w(x)\,\mathrm{d}x\quad\text{for each $A\in \mathcal{L}(\rn)$.}
$$
\end{defn}
\begin{defn}[Weighted and unweighted Lebesgue-Bochner space] {\it The Lebesgue-Bochner space}, denoted by $\lprne$, is defined as
$$
\lprne:=\{f:\rn\to E \,|\, f\text{ is Bochner measurable and } \left(\int\norm{f(x)}^p\, \mathrm{d} x\right)^{1/p}  < \infty\}.$$
We denote $\normlprne{f}:=(\int \norm{f(x)}^p\, \mathrm{d} x)^{1/p}$.

Let $w$ be a weight.   {\it The weighted Lebesgue-Bochner space}, denoted by $L^p_w(\rn\to E)$, is defined as
$$
L^p_w(\rn \to E):=\{f:\rn\to E \,|\, f\text{ is Bochner measurable and } \left(\int\norm{f(x)}^p\,w(x)\, \mathrm{d} x\right)^{1/p}  < \infty\}.$$
We denote $\normlpwv{f}:=(\int \norm{f(x)}^p\,w(x)\, \mathrm{d} x)^{1/p}$.
\end{defn}
\begin{defn}[Muckenhoupt weights]Let $M$ denote the Hardy-Littlewood maximal function. Suppose that $w$ is a weight function. 
Define {\it the dual weight function of $w$}, denoted by $\sigma_{w,p}$, by setting $\sigma_{w,p}(x):=w(x)^{-1/(p-1)}$ for each $x\in\rn$. We define the auxiliary quantities 

$$A_p(w;Q):=\frac{w(Q)}{\measure{Q}}\left(\frac{\sigma_{w,p}(Q)}{\measure{Q}}\right)^{p-1}\quad \text{for }p\in(1,\infty)
$$
and
$$
A_\infty(w;Q):=\frac{1}{w(Q)}\int_{Q}M(w1_Q).
$$
For $p\in(1,\infty]$ we define {\it the Muckenhoupt $A_p$ characteristic, denoted by $[w]_{A_p}$, of a weight $w$} by setting
$$
[w]_{A_p}:=\sup_{\text{all cubes } Q} A_p(w;Q),
$$
and we define {\it the Muckenhoupt's $A_p$ class}, denoted by $A_p$, as
$$
A_p:=\{w:\rn\to(0,\infty) \;| \;w \text{ is a weight and } [w]_{A_p}<\infty\}.
$$

\end{defn}
\begin{defn}[Vector-valued singular kernel] A function $K:\rn\times\rn\setminus \{(x,x):x\in \rn\} \to B(E\to E)$  is called {\it a singular kernel}, if and only if
\begin{enumerate}
\item[(i)]The function $K$ obeys the decay estimate $$\normbee{K(x,y)}\lesssim \frac{1}{\abs{x-y}^d}\quad\text{whenever }x\neq y.$$
\end{enumerate}

\begin{enumerate}
\item[(ii)]The function $K$ obeys the H\"older-type estimates  $$\normbee{K(x,y)-K(x',y)}\lesssim \left(\frac{\abs{x-x'}}{\abs{x-y}}\right)^\alpha \frac{1}{\abs{x-y}^d}\quad\text{whenever }0<\abs{x-x'}<\frac{1}{2}\abs{x-y}$$
and
$$\normbee{K(x,y)-K(x,y')}\lesssim \left(\frac{\abs{y-y'}}{\abs{x-y}}\right)^\alpha \frac{1}{\abs{x-y}^d}\quad\text{whenever }0<\abs{y-y'}<\frac{1}{2}\abs{x-y}$$
for some H\"older exponent $0<\alpha\leq 1$.
\end{enumerate}

\end{defn}

\begin{defn}[Vector-valued Calder\'{o}n-Zygmund operator]\label{vvczo} Let $1<p<\infty$. A linear operator $T:\lprne\to\lprne$ is called {\it a vector-valued Calder\'{o}n-Zygmund operator}, if and only if
\begin{enumerate}
\item[(i)] $T:\lprne\to\lprne$ is bounded.
\item[(ii)] There exists a singular kernel $K:\rn\times\rn\setminus \{(x,x):x\in \rn\} \to B(E\to E)$ such that 
$$
Tf(x)=\int K(x,y)f(y)\,\mathrm{d}x
$$
for every strongly measurable, bounded, and compactly supported function $f:\rn\to E$ and for every $x$ that lies outside the support of $f$.
\end{enumerate}
\end{defn}
\begin{rem}We include the condition (i) as a part of the definition of an $E$-valued Calder\'{o}n-Zygmund operator. In case of many classes of operators the condition (i) is checked by using theorems such as \cite[Theorem 5]{bourgain1983b}, an $E$-valued $T1$ theorem \cite{figiel1989}, an $E$-valued $Tb$ theorem \cite{hytonen2008}, or an operator-valued Fourier multiplier theorem \cite{weis2001}. These theorems presume that $E$ has the UMD-property, which means that $E$-valued martingale difference sequences are unconditional in $\lprne$. Moreover, in the case of the Hilbert-transform, which is a prototype of a singular integral operator, for the fulfilment of the condition (i) it is not only sufficient \cite{burkholder1981} but also necessary \cite{bourgain1983} that the Banach-space $E$ has the UMD-property. 

\end{rem}
Next we define the dyadic model operators $A_{\cs,k}$ that dominate each vector-valued Calder\'{o}n-Zygmund operator. The operators $A_{\cs,k}$ are precisely the same dyadic model operators that dominate each Calder\'{o}n-Zygmund operator in the real-valued case \cite{arXiv:1202.2229}. 
\begin{defn}[Pairwise nearly disjoint collection]\label{nearlydisjoint}Let $0<\nu<1$. A collection $\cs$ of measurable sets is called {\it pairwise nearly disjoint (with the parameter $\nu$)}, if and only if
\begin{enumerate}
\item[(i)] For every $Q\in \cs$ there exists a measurable subset $E(Q)\subset Q$ such that $\measure{E(Q)}\geq\nu\measure{Q}$.
\item[(ii)] For every $Q\in \cs$ and $Q'\in \cs$ such that $Q\neq Q'$ we have $E(Q)\cap E(Q')=\emptyset$. 
\end{enumerate}
\end{defn}

\begin{defn}[Dyadic model operator $A_{\cs,k}$]\label{dyadicmodeloperators}Let $\cs$ be a collection of dyadic cubes. Let $Q^{(k)}$ denote the $k$:th ancestor of a dyadic cube $Q$. We define the dyadic model operator $A_{\cs,k}$ by 

$$
A_{\cs,k}g:=\sum_{Q\in\cs} 1_Q \fint_{Q^{(k)}} g \quad\text{for every Lebesgue measurable function $g:\rn\to[0,\infty)$}.
$$
\end{defn}
Recall that for each translation parameter $u\in\{0,\frac{1}{3},\frac{2}{3}\}^d$ we have the translated dyadic system 
$$\cd^u:=\{2^{-j}([0,1)^d+m+(-1)^ju) : j\in\mathbb{Z}, m\in\mathbb{Z}^d\}.$$
Next we state our main theorem.
\begin{thm}[Pointwise dyadic domination theorem for vector-valued Calder\'{o}n-Zygmund operators]\label{introductiondyadicdomination} Suppose that $f:\rn\to E$ is a strongly measurable, bounded, and compactly supported function. Let $Q^0$ be a cube that contains the support of $f$. Suppose that $T$ is a vector-valued Calder\'{o}n-Zygmund operator with the H\"older exponent $\alpha\in(0,1]$. 

Then for each translated dyadic system  $u\in\{0,\frac{1}{3},\frac{2}{3}\}^d$ and for each $k\in\mathbb{N}$ there exists a collection $\cs^u_k$ of dyadic cubes such that the collection $\cs^u_k$ is pairwise nearly disjoint and for almost every $x\in\rn$ we have
$$
1_{Q^0}(x)\norm{(Tf)(x){\color{red}}} \lesssim \sum_{u\in\{0,\frac{1}{3},\frac{2}{3}\}^d}\sum_{k=0}^{\infty} 2^{-\alpha k} (A_{\cs^u_k,k}\norm{f})(x).
$$
The implicit constant in the inequality depends only on the dimension $d$ and on the constants that are implicit in the definition of a vector-valued Calder\'{o}n-Zygmund operator. Each collection $\cs^u_k$ depends on $T$ and $f$.

\end{thm}
The proof of the theorem is deferred to Section \ref{section:dyadicdomination}. We can use the theorem as a method to transfer results about real-valued model operators $A_{\cs,k}$ into results about vector-valued Calder\'{o}n-Zygmund operators. In regard to estimating the $\normlpwrne{\,\cdot\,}$-norm, note that by definition if $f\in\lpwrne$, then $\norm{f}\in\lpwrnr$ and $$\normlpwrnr{(\norm{f})}=\normlpwrne{f}.$$ Hence by Theorem \ref{introductiondyadicdomination} for each $f\in\lpwrne$ we have that
\begin{equation*}
\label{substitution2}
\normlpwrne{1_{Q^0}Tf} \lesssim \sum_{u\in\{0,\frac{1}{3},\frac{2}{3}\}^d}\sum_{k=0}^{\infty} 2^{-\alpha k}\normlpwrnr{ (A_{\cs^u_k,k}\norm{f})}.
\end{equation*}
Now by using a known estimate for $\normnone{A_{\cs,k}}_{B\left(\lpwrnr\to\lpwrnr\right)}$ we obtain the following corollary, the proof of which is deferred to Section \ref{section:dyadicdomination}.
\begin{corollary}[$A_p$ theorem for vector-valued Calder\'{o}n-Zygmund operators]\label{vvapt}Let $(E,\norm{\,\cdot\,})$ be a Banach space.  Suppose that $T$ is a vector-valued Calder\'{o}n-Zygmund operator with the H\"older exponent $\alpha\in(0,1]$. Then we have that 
\begin{equation}\label{extrapolation}
\normlpwrne{Tf}\lesssim \charap^{\max \{1,1/(p-1)\}}\normlpwrne{f}.
\end{equation}
for all $f\in\lpwrne$ and for all $w\in A_p$. The implicit constant in the inequality depends only on the dimension $d$ and on the constants that are implicit in the definition of a vector-valued Calder\'{o}n-Zygmund operator.
\end{corollary}
\begin{rem}By the sharp version of Rubio de Francia's extrapolation theorem \cite{grafakos2005}, version which extends with the same proof for Banach space valued functions, we have that the weighted norm estimate \eqref{extrapolation} for $p=2$ implies the weighted norm estimate \eqref{extrapolation} for every $p\in(1,\infty)$. However, we prove the weighted norm estimate \eqref{extrapolation} for every $p\in(1,\infty)$ directly, without using the sharp version of Rubio de Francia's extrapolation theorem. \end{rem}

The dyadic domination theorem for real-valued Calder\'{o}n-Zygmund operators \cite{arXiv:1202.2229} is based on Lerner's local oscillation decomposition formula \cite{lerner2010,arXiv:1202.1860}. Similarly, the dyadic domination theorem for vector-valued Calder\'{o}n-Zygmund operators is based on a vector-valued generalization of Lerner's local oscillation decomposition formula, to which we turn next. First we recall the notion of a local oscillation.

\begin{defn}[Local oscillation $\omega_\lambda(f;Q)$]Let $f:\rn\to E$ be a Lebesgue measurable function. Suppose that $Q\subset \rn$ is a Lebesgue measurable set. Let $0<\lambda<\frac{1}{2}$. {\it The local oscillation} or {\it the optimal oscillatory bound of $f$ on $Q$ with a $\lambda$-portion disregarded}, denoted by $\omega_\lambda(f;Q)$, is defined as

$$\omega_\lambda(f;Q):=\inf_{c\in E} \left( \min\left\{r\geq 0 : \frac{\measure{\{x\in Q : f(x)\notin B_E(c,r)\}}}{\measure{Q}}\leq\lambda\right\}\right).$$
\end{defn}
\begin{rem}Recall that {\it the decreasing rearrangement of $f$}, denoted by $\norm{f}^*$, is defined as $\norm{f}^*(t):=\min \{r\geq 0: \measure{\{x\in\rn : \norm{f(x)}>r\}}\leq t\}$. Note that $\omega_\lambda(f;Q):=\inf_{c\in E}\left(\norm{f-c}1_Q\right)^*(\lambda\measure{Q})$.

\end{rem}

Next we formulate the definition of a quasi-optimal center of oscillation, which is a vector-valued counterpart of a median. In Section \ref{sec:lernersformula} we show that there always exists a quasi-optimal center of oscillation and discuss the idea behind the notion.

\begin{defn}[Quasi-optimal center of oscillation $c_\lambda(f;Q)$]\label{oco}
\label{optimalcenterofoscillation} Let $f:\rn\to E$ be a Lebesgue measurable function. Suppose that $Q\subset \rn$ is a Lebesgue measurable set. Let $0<\lambda<\frac{1}{2}.$
{\it A quasi-optimal center of oscillation of $f$ on $Q$ with a $\lambda$-portion disregarded} or {\it a $\lambda$-pseudomedian of $f$ on $Q$}, denoted by $c_\lambda(f;Q)$, is defined as any vector $c\in E$ such that
$$\min\left\{r\geq 0 : \frac{\measure{\{x\in Q : f(x)\notin B_E(c,r)\}}}{\measure{Q}}\leq\lambda\right\}\leq 2\omega_\lambda(f;Q).
$$
\end{defn}
Endowed with the vector-valued generalization of a median, we now state the vector-valued generalization of Lerner's local oscillation decomposition formula.

\begin{thm}[Vector-valued generalization of Lerner's local oscillation decomposition formula]\label{lernersformula} Suppose that $(E,\norm{\,\cdot\,})$ is a Banach space. Let $\nu\in(0,1)$. 
Suppose that $f:\rn\to E$ is a strongly measurable function. Let $Q^0$ be a cube. 

Then there exists a collection $\cs$ of dyadic subcubes of $Q^0$  such that the collection $\cs$ is pairwise nearly disjoint with the parameter $\nu$ and for almost every $x\in Q^0$ and for every quasi-optimal center of oscillation $c_{1/4}(f;Q^0)$ we have
\be
\norm{f(x)-c_{1/4}(f;Q^0)}\leq 12\sum_{Q\in\cs}\omega_{(1-\nu)2^{-d-2}}(f;Q)1_{Q}(x).
\ee
\end{thm}

The proof of the Theorem \ref{lernersformula}, together with a discussion of the notion of the optimal oscillatory bound and the notion of a quasi-optimal center of oscillation, constitutes Section \ref{sec:lernersformula}.

\section{Pointwise dyadic domination theorem and $A_p$ theorem for vector-valued Calder\'{o}n-Zygmund operators}\label{section:dyadicdomination}
\begin{notation}Let $Q$ be a cube. We denote by $l_Q$ the side length and by $c_Q$ the center of the cube $Q$. Let $2^kQ$ denote the cube that has the same center as $Q$ but that has the side length $2^k l_Q$. We denote by $Q^{(k)}$ the $k$:th dyadic ancestor of a cube $Q$. We denote by $\cd$ the standard dyadic system,
$$\cd:=\{2^{-j}\left([0,1)^d+m\right) : j\in\mathbb{Z}, m\in\mathbb{Z}^d\}.$$
We use the notations $$\{\norm{f}>r\}:=\{x\in \rn : \norm{f(x)}>r\}$$ and
$$
\fint_Q f:=\frac{1}{\measure{Q}}\int_Q f.
$$
\end{notation}
\begin{thm}[Pointwise dyadic domination theorem for vector-valued Calder\'{o}n-Zygmund operators]\label{dyadicdomination} Suppose that $f:\rn\to E$ is a strongly measurable, bounded, and compactly supported function. Let $Q^0$ be a cube that contains the support of $f$. Suppose that $T$ is a vector-valued Calder\'{o}n-Zygmund operator with the H\"older exponent $\alpha\in(0,1]$. 

Then for each translated dyadic system  $u\in\{0,\frac{1}{3},\frac{2}{3}\}^d$ and for each $k\in\mathbb{N}$ there exists a collection $\cs^u_k$ of dyadic cubes such that the collection $\cs^u_k$ is pairwise nearly disjoint and for almost every $x\in\rn$ we have
$$
1_{Q^0}(x)\norm{(Tf)(x){\color{red}}} \lesssim \sum_{u\in\{0,\frac{1}{3},\frac{2}{3}\}^d}\sum_{k=0}^{\infty} 2^{-\alpha k} (A_{\cs^u_k,k}\norm{f})(x).
$$
The implicit constant in the inequality depends only on the dimension $d$ and on the constants that are implicit in the definition of a vector-valued Calder\'{o}n-Zygmund operator. Each collection $\cs^u_k$ depends on $T$ and $f$.
\end{thm}
\begin{proof} 
The proof of the dyadic domination theorem in the vector-valued case proceeds parallel to the proof in the real-valued case \cite{arXiv:1202.2229}. By using the vector-valued generalization of Lerner's local oscillation decomposition formula, Theorem \ref{lernersformula}, we dominate $T$ by a series of optimal oscillatory bounds $\omega_\lambda(Tf;Q)1_Q$ summed over all the cubes of a pairwise nearly disjoint (with the parameter $\frac{1}{2}$) collection $\cs'$ of dyadic cubes,
$$
1_{Q^0}(x)\norm{Tf(x)-c_{1/4}(Tf;Q^0)}\leq 12\sum_{Q\in\cs'}\omega_{2^{-d-3}}(Tf;Q)1_{Q}(x).
$$

Let $\lambda\in(0,\frac{1}{2})$. We claim that for each cube $Q$ the optimal oscillatory bound $\omega_\lambda(Tf;Q)$ is bounded by a series of the non-dyadic integral averages $\fint_{2^k Q} \norm{f}$, 
$$
\omega_\lambda(Tf;Q)\lesssim_{T,d,\lambda} \sum_{k=0}^\infty 2^{-\alpha k} \fint_{2^kQ}\norm{f}.
$$
Moreover, we claim that for each cube $Q^0$ that contains the support of $f$ we have 
$$
\norm{c_\lambda(Tf;Q^0)}\lesssim_{T,\lambda} \fint_{Q^0}\norm{f}.
$$
Assuming for the moment these claims, which are stated and proven as Lemma \ref{oscillationcalderon}, we complete the proof.  

Recall that for each translation parameter $u\in\{0,\frac{1}{3},\frac{2}{3}\}^d$ there is the translated dyadic system 
$$\cd^u:=\{2^{-j}([0,1)^d+m+(-1)^ju) : j\in\mathbb{Z}, m\in\mathbb{Z}^d\},$$ and that for each cube $Q$ and each $k\in\mathbb{N}$ we can find  a cube $R(Q,k)$ that is dyadic in the translated dyadic systems $\cd^{u(Q,k)}$ for some $u(Q,k)\in\{1,\frac{1}{3},\frac{2}{3}\}^d$ and that satiesfies $Q\subset R(Q,k),\; 2^kQ\subset R(Q,k)^{(k)}$, and $3\cdot l_Q < l_{R(Q,k)} \leq 6\cdot l_Q$ \cite[Proposition 2.5.]{arXiv:1202.2229}. Hence we can dominate each non-dyadic integral average $1_Q\fint_{2^k Q}\norm{f}$ by the dyadic integral average $1_{R(Q,k)}\fint_{R(Q,k)^{(k)}} \norm{f}$,
$$
1_Q\fint_{2^k Q}\norm{f}\lesssim_d 1_{R(Q,k)}\fint_{R(Q,k)^{(k)}} \norm{f}.
$$

Let $k\in\mathbb{N}$. By the definition of a pairwise nearly disjoint collection with the parameter $\frac{1}{2}$, for each $Q\in\cs'$ there exists a measurable subset $E(Q)\subset Q$ such that $\measure{E(Q)}\geq\frac{1}{2}\measure{Q}$ and for each $Q\in\cs'$ and $Q'\in \cs'$ such that $Q\neq Q'$ we have $E(Q)\cap E(Q')=\emptyset$.  Since $Q$ and $R(Q,k)$ are of comparable size, we have $$\measure{E(Q)}\geq\frac{1}{2}\measure{Q}=\frac{1}{2} (l_Q)^d\geq \frac{1}{2}\left( \frac{l_{R(Q,k)}}{6}\right)^d= \frac{1}{2\cdot 6^d} \measure{R(Q,k)}.$$ 
Hence the collection $\cs_k:=\{R(Q,k)\}_{Q \in \cs'}$ is pairwise nearly disjoint with the parameter $\frac{1}{2\cdot 6^d}$. Moreover, observe that since the side-lengths $l_Q$ and $l_{R(Q,k)}$ are both powers of $2$ and since $3\cdot l_Q < l_{R(Q,k)} \leq 6\cdot l_Q$, we have in fact that $l_{R(Q,k)}=4l_Q$. Since a cube $R(Q,k)$ with side-length $4l_Q$ can contain at most $4^d$ cubes of side lenth $l_Q$ and since $\cs_k:=\{R(Q,k)\}_{Q\in\cs'}$, we have
$$
\sum_{Q\in\cs'} 1_{R(Q,k)} \leq 4^d \sum_{R\in\cs_k} 1_{R}.
$$
Observe that we can decompose the pairwise nearly disjoint collection $\cs_k$, which contains dyadic cubes of various translated dyadic systems, into $3^d$ pairwise nearly disjoint collections $\cs^u_k$, each of which contains  dyadic cubes of at most one translated dyadic system, 
$$
\cs_k=\{R(Q,k)\}_{Q \in \cs'} =\bigcup_{u\in\{0,\frac{1}{3},\frac{2}{3}\}^d} \{R(Q,k)\}_{Q \in \cs' \text{and} R(Q,k)\in\cd^u}=:\bigcup_{u\in\{0,\frac{1}{3},\frac{2}{3}\}^d} \cs^u_k.
$$ 

Altogether we have obtained that
\ba
1_{Q^0}(x)\norm{Tf(x)}&\leq& 1_{Q^0}(x) \norm{Tf(x)-c_{1/4}(Tf;Q^0)}+1_{Q^0}(x)\norm{c_{1/4}(Tf;Q^0)}\\
&\lesssim_{T} & \sum_{Q\in\cs'}\omega_{2^{-d-3}}(f;Q)1_{Q}(x)+1_{Q^0}(x)\fint_{Q^0}\norm{f}\\
&\lesssim_{T,d}& \sum_{k=0}^\infty 2^{-\alpha k} \sum_{Q\in\cs'}1_Q(x)\fint_{2^kQ}\norm{f}\\
&\lesssim_{d}&\sum_{k=0}^\infty 2^{-\alpha k} \sum_{u\in\{0,\frac{1}{3},\frac{2}{3}\}^d}\underbrace{\sum_{R\in\cs^u_k} 1_{R}(x)\fint_{R^{(k)}} \norm{f}}_{=:(A_{\cs^u_k,k}\norm{f})(x)}.
\ea
This concludes the proof, modulo Lemma \ref{oscillationcalderon}.
\end{proof}
\begin{lemma}\label{oscillationcalderon}

Suppose that $\lambda\in(0,\frac{1}{2}).$ Let $L^\infty_c(\rn\to E)$ denote the set of all strongly measurable, bounded, and compactly supported functions from $\rn$ to $E$. Suppose that $T$ is a vector-valued Calder\'{o}n-Zygmund operator with the H\"older exponent $\alpha\in(0,1]$.

Then we have that

$$\omega_\lambda(Tf;Q)\lesssim\; \sum_{k=0}^\infty 2^{-\alpha k} \fint_{2^kQ}\norm{f}$$
for all $f\in L^\infty_c(\rn\to E)$ and all cubes $Q$. Furthermore, we have that

$$\norm{c_\lambda(Tf;Q^0)}\lesssim\; \fint_{Q^0} \norm{f}
$$
for all $f\in L^\infty_c(\rn\to E)$ and all cubes $Q^0$ that contain the support of $f$.

\end{lemma}

\begin{proof} We shall first proof the estimate for $\omega_\lambda(Tf;Q)$. Recall that the Euclidean distance $\abs{x}:=\left(\sum_{i=1}^dx_i^2\right)^{1/2}$ and the distance $\abs{x}_\infty:=\max_{i=1,\ldots,d} \abs{x_i}$ are equivalent in the sense that $\abs{x}_\infty\leq\abs{x}\leq d^{1/2}\,\abs{x}_\infty$. From the equivalence of the distances it follows that if a singular kernel $K$ satisfies the decay and the H\"older estimates of its definition for the Euclidean distance $\abs{x}$, then $K$ satisfies the same estimates (with the implicit constants depending on $d$) for the distance $\abs{x}_\infty$. For this proof we shall work with the distance  $\abs{x}_\infty$ in order to slightly simplify the use of the kernel estimates. Let $f:\rn\to E$ be a strongly measurable, bounded, and compactly supported function. Suppose that $Q$ is a cube.  Let $x\in Q$. Since
\begin{equation*}
\te f(x)-\te (1_{(2Q)^c}f)(c_Q)=\te (1_{2Q}f)(x) + \left(\te (1_{(2Q)^c}f)(x) - \te (1_{(2Q)^c}f)(c_Q)  \right),
\end{equation*}
we have
\begin{equation}
\label{oscillationlemma1}
\norm{Tf(x)-T(1_{(2Q)^c}f)(c_Q)}\leq \norm{T (1_{2Q}f)(x)} + \norm{T(1_{(2Q)^c}f)(x) - T(1_{(2Q)^c}f)(c_Q) }.
\end{equation}
Observe that for the last term on the right-hand side we can use the integral presentation with a singular kernel, because $x\in Q$ lies outside the support of $1_{(2Q)^c}f$, and we can use the H\"older estimate of a singular kernel, because for $y\in (2Q)^c$ we have $\abs{c_Q-x}_\infty\leq \frac{1}{2}\abs{c_Q-y}_\infty$. Therefore
\begin{eqnarray}
\label{oscillationlemma2}
\norm{\te (1_{(2Q)^c}f)(x) - \te (1_{(2Q)^c}f)(c_Q)  }&\leq& \sum_{k=1}^\infty \int_{2^{k+1}Q\setminus 2^kQ} \norm{(K(x,y)-K(c_Q,y)f(y)}\,\mathrm{d}y\nonumber\\
&\lesssim_{T,d}& \sum_{k=1}^\infty \int_{2^{k+1}Q\setminus 2^kQ} \underbrace{\left(\frac{\abs{x-c_Q}_\infty}{\abs{y-c_Q}_\infty}\right)^\alpha}_{\lesssim 2^{-\alpha k}} \underbrace{\left(\frac{1}{\abs{y-c_Q}_\infty^d}\right)}_{\lesssim {\left(2^kl(Q)\right)}^{-d}\eqsim_d\measure{2^{k+1}Q}^{-1} } \norm{f(y)}\,\mathrm{d}y\nonumber\\
&\lesssim_{d} & \sum_{k=1}^\infty 2^{-\alpha k}\fint_{2^{k+1}Q} \norm{f(y)}\,\mathrm{d}y.
\end{eqnarray}
Let $g:\rn\to E$ be a Lebesgue measurable function. Recall that {\it the decreasing rearrangement of g}, denoted by $\norm{g}^*$, is defined as $$\norm{g}^*(t):=\min \{r\geq 0: \measure{\{\norm{g}>r\}}\leq t\} \quad \text{for every } t\in[0,\infty).$$ In order to estimate the optimal oscillatory bound $\omega_\lambda(f;Q)$ by exploiting the properties of a decreasing rearrangement, we write the optimal oscillatory bound in terms of the decreasing rearrangement $$\omega_\lambda(f;Q)=\inf_{c\in E}\left(\norm{f-c}1_Q\right)^*(\lambda\measure{Q}).$$ 
We shall use the following properties, which are all well-known, of a decreasing rearrangement.
\begin{itemize}
\item[(i)]$\norm{g}^*(t)\leq t^{-1/p}\normweaklprne{g}\quad\text{for each } p\in(0,\infty).$
\item[(ii)]$ \norm{g+v}^*(t) \leq \norm{g}^*(t)+\norm{v}$ for every constant vector  $v\in E$.
\item[(iii)]If for some real constant $C\in[0,\infty)$ we have that $\norm{g(x)}\leq C\norm{f(x)}$ for almost every point $x\in\rn$, then $\norm{g}^*(t)\leq C \norm{f}^*(t)$ for every $t\in[0,\infty)$.
\end{itemize}
By the properties (ii) and (iii) of a decreasing rearrangement and the inequalities $\eqref{oscillationlemma1}$ and $  \eqref{oscillationlemma2}$ we can estimate the optimal oscillatory bound $\omega_\lambda(Tf;Q)$ as follows.
\ba
\omega_\lambda(Tf;Q)&=&\inf_{c\in E}\left(\norm{Tf-c}1_Q\right)^*(\lambda\measure{Q})\\
&\leq& \left(\norm{Tf(x)-T(1_{(2Q)^c}f)(c_Q)}1_Q\right)^*(\lambda\measure{Q})\\
&\lesssim_{T,d}&\left(\left(\norm{T(1_{2Q}f)} +\sum_{k=1}^\infty 2^{-\alpha k}\fint_{2^{k+1}Q} \norm{f}  \right)1_Q\right)^*(\lambda\measure{Q})\\
&\leq& \left(\norm{T(1_{2Q}f)}1_Q\right)^*(\lambda\measure{Q})+ \sum_{k=1}^\infty 2^{-\alpha k}\fint_{2^{k+1}Q} \norm{f}.
\ea
Recall the weak-type (1,1) inequality for $T$, $$\normweaklonerne{Tg}\lesssim \normlonerne{g}\quad\text{for all $g\in\lonernr$},$$
which is well-known and proven in \cite[The proof of Theorem 1]{benedek1962}.
By the property (i) of a decreasing rearrangement together with the weak-type (1,1) inequality for $T$ we have 
$$
\left(\norm{T(1_{2Q}f)}1_Q\right)^*(\lambda \measure{Q})\leq \frac{1}{\lambda\measure{Q}}\normweaklonerne{T(1_{2Q}f)}\lesssim_{T,\lambda} \frac{1}{\measure{2Q}}\normlonerne{1_{2Q}f}=\fint_{2Q}\norm{f}.
$$
Hence altogether we have
$$
\omega_\lambda(Tf;Q)\lesssim_{T,d,\lambda} \sum_{k=0}^\infty 2^{-\alpha k}\fint_{2^{k}Q} \norm{f}.
$$

Next we prove the estimate for $\norm{c_\lambda(Tf;Q^0)}$. By the inequality \eqref{3rpropertyb} of Lemma \ref{3rproperty}, the property (i) of a decreasing rearrangement, and the weak-type (1,1) inequality for $T$ we have that
\ba
\norm{c_\lambda(Tf;Q)}&\leq& 3\left(\norm{Tf}1_{Q}\right)^*(\lambda \measure{Q})\\
&\leq&3\frac{1}{\lambda\measure{Q}}\normweaklonerne{Tf}\\
&\lesssim_{T,\lambda}& \frac{1}{\measure{Q}}\normlonerne{f}.
\ea
If $Q^0$ is a cube that contains the support of $f$, then
$$
\frac{1}{\measure{Q^0}}\normlonerne{f}=\frac{1}{\measure{Q^0}}\normlonerne{1_{Q^0}f}=\fint_{Q^0} \norm{f}.
$$

\end{proof}

Next we state the vector-valued $A_p$ theorem as a corollary of the vector-valued pointwise dyadic domination theorem.
\begin{corollary}[$A_p$ theorem for vector-valued Calder\'{o}n-Zygmund operators] Let $E$ be a Banach space.  Suppose that $T$ is a vector-valued Calder\'{o}n-Zygmund operator on the Lebesgue-Bochner space $\lprne$. Define a two-weight generalization $[w,\sigma]_{A_p}$ of the Muckenhoupt $A_p$ characteristic by
$$
[w,\sigma]_{A_p}:=\sup_{\text{all cubes } Q} \frac{w(Q)}{\measure{Q}}\left(\frac{\sigma(Q)}{\measure{Q}}\right)^{p-1}.
$$

Then we have that 
\begin{equation}\label{a2first}
\normlpwrne{T(f\sigma)}\lesssim [w,\sigma]_{A_p}^{1/p}\left([w]_{A_\infty}^{1-1/p}+[\sigma]_{A_\infty}^{1/p}\right)\normlpsrne{f}
\end{equation}
for all $f\in\lpsrne$, for all $w\in A_\infty$, and for all $\sigma\in A_\infty$. In particular we have that
\begin{equation}\label{a2second}
\normlpwrne{Tf}\lesssim [w]_{A_p}^{\max \{1,1/(p-1)\} } \normlpwrne{f}
\end{equation}
for all $f\in\lpwrne$ and for all $w\in A_p$. The implicit constants in the inequalities \eqref{a2first} and \eqref{a2second} depend only on the dimension $d$ and on the constants that are implicit in the definition of a vector-valued Calder\'{o}n-Zygmund operator.  
\end{corollary}
\begin{proof}

 Recall that an operator that is bounded on $\lprnr$ and that has the form
\begin{equation}\label{pdsform}
g\rightmapsto \sum_{Q\in\cd }  \sum_{\substack{R\in\cd, S\in\cd\\R\subset Q, S\subset Q\\l(R)=2^{-m} l(Q), l(S)=2^{-n}l(Q)}} a_{QRS}\, 1_R \fint_{S} g ,
\end{equation}
for some non-negative coefficients $0\leq a_{QRS}\leq 1$, and non-negative integers $m\in\mathbb{N}$ and $n\in\mathbb{N}$,  is called {\it  a positive dyadic shift operator of complexity $\max\{1,m,n\}$}. In particular the operator $A_{\cs,k}$ is a positive dyadic shift operator of complexity $k$, because by \cite[Proposition 2.6]{arXiv:1202.2229} the operator $A_{\cs,k}$ is bounded on $\lprnr$ and because by definition the operator $A_{\cs,k}$ has the form \eqref{pdsform} for $m=k, n=0$, and $$
a_{QRS}=\begin{cases}1 &\text{if } (R,S) \in \{(R,R^{(k)})\}_{R\in\cs},\\
0 & \text{otherwise.} 
\end{cases}.$$
We use the shorthand $N_p(w,\sigma):= [w,\sigma]_{A_p}^{1/p}\left([w]_{A_\infty}^{1-1/p}+[\sigma]_{A_\infty}^{1/p}\right)$. By \cite[Proposition 3.2 and Proposition 5.3.]{arXiv:1106.4797}, if $S_k$ is a positive dyadic shift operator of complexity $k$, then 
\begin{equation*}\label{positivedyadicshiftestimate}
\normlpwrnr{S_k(g\sigma)}\lesssim (1+k)N_p(w,\sigma)\normlpsrnr{g}.
\end{equation*}
Hence in particular
\begin{equation}
\label{weightedpositive}
\normlpwrnr{A_{\cs,k}(g\sigma)}\lesssim (1+k) N_p(w,\sigma)\normlpsrnr{g}.
\end{equation}

Since the subspace of strongly measurable, bounded, and compactly supported functions is dense in $\lpsrne$, since $T$ is linear and since $\lpwrne$ is complete, it suffices to prove the inequality \eqref{a2first} for every such function. Suppose that $f:\rn\to E$ is a strongly measurable, bounded, and compactly supported function. Let $Q^0$ be a cube that contains the support of $f$. By the pointwise dyadic domination theorem, Theorem \ref{dyadicdomination}, and the inequality \eqref{weightedpositive} for $g=\norm{f}$ we obtain
$$
\normlpwrne{1_{Q^0}T(f\sigma)} \lesssim_{T,d}  N_p(w,\sigma)\normlpsrne{f}.
$$
Choosing a sequence of cubes $Q^0_N$ such that each $Q^0_N$ contains the support of $f$, $ Q^0_N\subset Q^0_{N+1}$, and $\bigcup_N Q^0_N =\rn$, and applying the monotone convergence theorem we obtain the weighted inequality \eqref{a2first}.

It is well-known that the inequality \eqref{a2first} implies the inequality \eqref{a2second} as follows. Choosing $\sigma=w^{1/(1-p)}=:\sigma_{w,p}$ and $f\sigma_{w,p}=h$ in the inequality \eqref{a2first} we obtain
\begin{equation}
\label{extens}
\normlpwrne{Th} \lesssim_{T,d} N_p(w,\sigma_{w,p}) \normlpwrne{h}.
\end{equation}
From the well-known facts $[w,\sigma_{w,p}]_{A_p}=[w]_{A_p}$, $[w]_{A_\infty}\leq [w]_{A_{p}}$,       $[\sigma_{w,p}]_{A_{p/(p-1)}}=[w]_{A_p}^{1/(p-1)}$,  and $1\leq[w]_{A_p}$ for every $p\in(1,\infty)$ and for every $w\in A_p$, it follows that 
$$
N_p(w,\sigma_{w,p})\leq [w]_{A_p}^{1/p}\left([w]_{A_p}^{1-1/p}+[\sigma_{w,p}]_{A_{p/(p-1)}}^{1/p}\right)\leq 2 [w]_{A_p}^{\max \{1,1/(p-1)\} }.
$$

\end{proof}

\section{Local oscillation decomposition for vector-valued functions}\label{sec:lernersformula}

\begin{notation} Denote $\mathbb{N}:=\{0,1,2,3,\ldots\}$ the set of all natural numbers. Let $f:\rn\to E$ be a function. We use the notation 
$$
\{f\notin B_E(c,r)\}:=\{x\in \rn : f(x)\notin B_E(c,r)\}=\{x\in \rn : \norm{f(x)-c}>r\}=:\{\norm{f-c}>r\}.
$$
\end{notation}
Let us first recall the key notions of Lerner's local oscillation decomposition formula \cite{lerner2010,arXiv:1202.1860}.

\begin{defn} Let $g:\rn\to \mathbb{R}$ be a Lebesgue measurable function. 
{\it The local oscillation of $g$ on $Q$}, denoted by $\omega_\lambda(g;Q)$, is defined as
\be
\omega_\lambda(g;Q):=\inf_{c\in\mathbb{R}} \min \{r\geq 0: \frac{\measure{\{x\in Q : \abs{g(x)-c}>r\}}}{\measure{Q}}\leq \lambda\}.
\ee
\end{defn}
\begin{defn} Let $g:\rn\to \mathbb{R}$ be a Lebesgue measurable function. 
{\it A median of $g$ on $Q$}, denoted by $m(g;Q)$, is defined as any real number $m(g;Q)$ such that both
\be
\frac{\measure{\{x\in Q : g(x)>m(f;Q)\}}}{\measure{Q}}\leq \frac{1}{2}
\ee
and
\be 
\frac{\measure{\{x\in Q : g(x)<m(f;Q)\}}}{\measure{Q}}\leq \frac{1}{2}.
\ee
\end{defn}
Recall Lerner's local oscillation decomposition formula.
\begin{thm}[Lerner's local oscillation decomposition formula]\label{lernersoriginal}  
Suppose that $g:\rn\to \mathbb{R}$ is a Lebesgue measurable function. Let $Q^0$ be a cube. 

Then there exists a (possibly empty) collection $\cs $ of dyadic subcubes of $Q^0$ such that the collection $\cs$ is pairwise nearly disjoint with the parameter $\frac{1}{2}$ and for every median $m(g;Q^0)$ and for almost every $x\in Q^0$ we have
\be
\abs{g(x)-m(g;Q^0)}\leq 4 \sup_{ \substack{Q\subset Q^0, \\Q\text{ dyadic cube}} } \omega_{2^{-d-2}}(g;Q)1_Q(x) +2\sum_{Q\in \cs}\omega_{2^{-d-2}}(g;Q)1_{Q}(x).
\ee
\end{thm}
\begin{rem} The collection $\cs$ in the statement of Theorem \ref{lernersoriginal} is in fact pairwise nearly disjoint with $E(Q):=Q\setminus\bigcup_{Q' : Q'\subsetneq Q} Q'$ as the set $E(Q)$ in the definition of a pairwise nearly disjoint collection, Definition \ref{nearlydisjoint}. 
\end{rem}
We are seeking a generalization of Lerner's local oscillation decomposition that would hold for vector-valued functions. The notion of the local oscillation $\omega_\lambda(f;Q)$ carries over for a vector-valued function simply by regarding the real line $(\mathbb{R},\abs{\,\cdot\,})$ as a particular instance of a Banach space $(E,\norm{\,\cdot\,})$. 

Behind the notion of the local oscillation is the following idea. We try to cover with a ball $B_E(c,r)$ the values that a function $f|_Q:Q\to E$ attains over a cube $Q$. We try to choose the center $c$ and the radius $r$ so that the radius $r$ is as small as possible. (If we regard the domain $Q$ as the time and the target space $E$ as the space, then the center $c$ can be thought of as the center of oscillation and the radius $r$ as the maximum of the amplitude about the center $c$ of oscillation over time $Q$.) How big do we expect the radius $r$ to be? It depends on whether we consider the whole cube or only a part of it. On one extreme, if we do not disregard any part of the cube and consider the image of the whole cube, then the radius $r$ is likely to be very large, and on the other extreme, if we disregard the whole cube and consider the image of an empty set, then the radius $r$ is zero. Therefore there is a trade off between the radius $r$ of the ball and the portion of the cube that we disregard. (The maximum of the amplitude over just a moment may be small, whereas the maximum of the amplitude over all times may be large.) To make the notion precise we introduce the following definition.

\begin{defn}
\label{leastbound}
{\it The least bound of $f$ (about the origin) on $Q$ with a $\lambda$-portion disregarded}, denoted by $\thesymbolforalpha_\lambda(f;Q)$, is defined as
\be
\thesymbolforalpha_\lambda(f;Q):=\min\{r\geq 0 : \frac{\measure{Q\cap\{f\notin B_E(0,r)\}}}{\measure{Q}}\leq\lambda\}.
\ee

\end{defn}

\begin{rem}
Recall that {\it the decreasing rearrangement of $f$}, denoted by $\norm{f}^*$, is defined by $\norm{f}^*(t):=\min \{\thesymbolforalpha\geq 0: \measure{\{\norm{f}>\thesymbolforalpha\}}\leq t\}$. Note that $$\thesymbolforalpha_\lambda(f;Q)=\norm{f1_Q}^*(\lambda\abs{Q}).$$
\end{rem}

\begin{rem}
Observe that 
\be
\thesymbolforalpha_\lambda(f-c;Q)=\min\{r\geq 0 : \frac{\measure{Q\cap\{f\notin B_E(c,r)\}}}{\measure{Q}}\leq\lambda\}
\ee
and
\be
\omega_\lambda(f;Q)=\inf_{c\in E} \thesymbolforalpha_\lambda(f-c;Q).
\ee
Hence it is natural to call $\thesymbolforalpha_\lambda(f-c;Q)$ {\it the least bound of $f$ about the center $c$ on $Q$ with a $\lambda$-portion disregarded} and $\omega_\lambda(f;Q)$ {\it the optimal oscillatory bound of $f$ on $Q$ with a $\lambda$-portion disregarded}.  
\end{rem}

The median and its properties are pivotal in the proof of Lerner's local oscillation decomposition formula \cite{lerner2010,arXiv:1202.1860}. The notion of a median is based on the ordering of the real numbers. Therefore we do not know {\it a priori}, whether the notion of a median has a vector-valued counterpart. However, we recall the following suggestive lemma \cite[Equation (4.2)]{arXiv:1202.1860}.
\begin{lemma}
\label{medianoptimal}
Let $\lambda\in[0,\oh)$. Suppose that $g:\rn\to\mathbb{R}$ is Lebesgue measurable. Then 
\be
\thesymbolforalpha_\lambda(g-m(g;Q);Q)\leq 2 \omega_\lambda(g;Q)
\ee
\end{lemma}
Recall the discussion preceding Definition \ref{leastbound}. The geometric content of Lemma \ref{medianoptimal} is as follows. Suppose that we try to cover with a ball (with a radius as small as possible) the image of a cube (under a function). Assume that we allow the image of a tiny portion of the cube to lie outside of the ball. Then the ball centered at a median has a radius that is within a multiple of $2$ of the optimal radius. Inspired by this observation we define a vector-valued counterpart of a median as follows.
\begin{defn} Let $f:\rn\to E$ be a Lebesgue measurable function. 
{\it A quasi-optimal center of oscillation of $f$ on $Q$ with a $\lambda$-portion disregarded} or {\it a $\lambda$-pseudomedian of $f$ on $Q$}, denoted by $c_\lambda(f;Q)$, is defined as any vector $c_\lambda(f;Q)\in E$ such that
\be
\thesymbolforalpha_\lambda(f-c_\lambda(f;Q);Q)\leq 2 \omega_\lambda(f;Q)
\ee
\end{defn}
Next we check that for each Lebesgue measurable vector-valued function there exists a quasi-optimal center of oscillation.
If $\omega_\lambda(f;Q)>0$, then by the definition of the greatest lower bound there exists a vector $c\in E$ such that $\thesymbolforalpha_\lambda(f-c;Q)\leq 2 \inf_{c\in E} \thesymbolforalpha_\lambda(f-c;Q) = 2\omega_\lambda(f;Q)$. If $\omega_\lambda(f;Q)=0$, then the existence of $c\in E$ such that $\thesymbolforalpha_\lambda(f-c;Q)=0$ is assured by the following lemma.

\begin{lem}\label{existenceoco} Let $\lambda\in[0,\oh)$. If $\omega_\lambda(f;Q)=0$, then there exists a vector $c\in E$ such that $\thesymbolforalpha_\lambda(f-c;Q)=0$.
\end{lem}
\begin{proof}
Suppose that $\omega_\lambda(f;Q)=0$. Recall that by definition $\omega_\lambda(f;Q)=\inf_{c\in E} \thesymbolforalpha_\lambda(f-c;Q)$. By the definition of the greatest lower bound there exists a sequence of vectors $(c_n)\subset E$ such that $\lim_{n\to\infty}\thesymbolforalpha_\lambda(f-c_n;Q)=0$. For further convenience let $\thesymbolforalpha_n:=\thesymbolforalpha_\lambda(f-c_n;Q)$. By the definition of $\thesymbolforalpha_\lambda(f-c_n;Q)$ we have $\measure{Q\cap\{\norm{f-c_n}>\thesymbolforalpha_n\}}\leq \lambda\measure{Q}$. This together with the assumption $0\leq\lambda<\oh$ implies that 
\be
\abs{Q\cap\{\norm{f-c_n}\leq \thesymbolforalpha_n\}\cap\{\norm{f-c_m}\leq \thesymbolforalpha_m\}}\geq (1-2\lambda)\abs{Q}>0.
\ee
It follows that there exists $x_0\in Q$ such that both $ \norm{f(x_0)-c_n}\leq \thesymbolforalpha_n$ and $ \norm{f(x_0)-c_m}\leq \thesymbolforalpha_m$. By the triangle inequality $\norm{c_n-c_m}\leq \norm{f(x_0)-c_n}+\norm{f(x_0)-c_m}\leq\thesymbolforalpha_n+\thesymbolforalpha_m.$ Therefore $(c_n)$ is a Cauchy sequence. Thus by the completeness of $E$ there exists $c\in E$ such that $\lim_{n\to\infty}\norm{c_n-c}=0$. For each integer $k$ choose an integer $n_k$ so large that both $\norm{c-c_{n_k}}\leq \frac{1}{2k}$ and $\thesymbolforalpha_{n_k}\leq\frac{1}{2k}$. Let $x\in\rn$. Now if $\norm{f(x)-c_{n_k}}\leq\thesymbolforalpha_{n_k}$, then $\norm{f(x)-c}\leq\norm{f(x)-c_{n_k}}+\norm{c-c_{n_k}}\leq \thesymbolforalpha_{n_k}+\frac{1}{2k}\leq \frac{1}{2k}+\frac{1}{2k} \leq \frac{1}{k}$. By contrapositive, if $\norm{f(x)-c}>\frac{1}{k}$, then $\norm{f(x)-c_{n_k}}>\thesymbolforalpha_{n_k}$. Altogether we have
\ba 
\measure{Q\cap\{\norm{f-c}>0\}}&=&\limsup_{k\to\infty}\measure{Q\cap\{\norm{f-c}>\frac{1}{k}\}}\nonumber\\
&\leq&\limsup_{k\to\infty}\measure{Q\cap\{\norm{f-c_{n_k}}>\thesymbolforalpha_{n_k}\}}\nonumber\\
&\leq&\lambda\abs{Q}.
\ea
This together with the definition $\thesymbolforalpha_\lambda(f-c)=\min\{\thesymbolforalpha\geq 0:\abs{Q\cap\{\norm{f-c}>\thesymbolforalpha\}}\leq\lambda\abs{Q}\}$ implies that $\thesymbolforalpha_\lambda(f-c) = 0$.
 \end{proof}

\begin{lemma}\label{3rproperty} Let $\lambda\in[0,\oh)$. Suppose that $\m$ is a quasi-optimal center of oscillation. Let $c\in E$ and $r>0$. Then if $\measure{Q\cap\{f\notin B_E(c,r)\}}\leq \lambda \measure{Q},$ then $\m \in B_E(c,3r).$

\end{lemma}
\begin{rem}
In the notation of a decreasing rearrangement, the lemma states that \begin{equation}\label{3rpropertya}\norm{c-\m}\leq 3\left(\norm{f-c}1_Q\right)^*(\lambda \measure{Q}).
\end{equation}
In particular for $c=0$ we obtain
\begin{equation}\label{3rpropertyb}
\norm{\m}\leq 3\left(\norm{f}1_Q\right)^*(\lambda \measure{Q}).
\end{equation} 
The equations \eqref{3rpropertyb} and \eqref{3rpropertya} are in fact equivalent because each quasi-optimal center of oscillation has the additivity property $c+c_\lambda(f;Q)=c_\lambda(f+c\,;Q)$, which follows from  $$\thesymbolforalpha_\lambda\left(f+c-(c+c_\lambda(f;Q));Q\right)=\thesymbolforalpha_\lambda(f-(c-c)-c_\lambda(f;Q);Q)\leq 2 \omega_\lambda(f;Q)=2\omega_\lambda(f+c;Q).$$ 
\end{rem}
\begin{proof}[Proof of Lemma \ref{3rproperty}] 
 Recall that by definition $$\thesymbolforalpha_\lambda(f-c;Q)=\min\{\thesymbolforalpha \geq 0: \measure{Q\cap\{\norm{f-c}>\thesymbolforalpha\}}\leq \lambda \measure{Q}\}$$ and that by definition $\omega_\lambda(f;Q)=\inf_{c'\in E}\thesymbolforalpha_\lambda(f-c';Q)$. Observe that the definitions and the assumption imply that $\omega_\lambda(f;Q)\leq\thesymbolforalpha_\lambda(f-c;Q)\leq r$. Recall that by the definition of $c_\lambda(f;Q)$ we have $\thesymbolforalpha_\lambda(f-c_\lambda(f;Q);Q)\leq 2\omega_\lambda(f;Q)$. Thus altogether we have $\thesymbolforalpha_\lambda(f-c_\lambda(f;Q);Q)\leq 2r$. This together with the definition of $\thesymbolforalpha_\lambda(f-c_\lambda(f;Q);Q)$ implies that 
 \be
 \measure{Q\cap\{\norm{f-c_\lambda(f;Q)}>2r\}}\leq  \lambda\abs{Q}.
 \ee
From this and the assumptions it follows that

 \be
  \measure{Q\cap\{\norm{f-c}\leq r\}\cap\{\norm{f-c_\lambda(f;Q)}\leq 2r\}}\geq (1-2\lambda)\measure{Q}>0
  \ee
Thus there exists $x_0\in Q$ such that both $\norm{f(x_0)-c}\leq r$ and $\norm{f(x_0)-\m}\leq 2r$. Hence by the triangle inequality $\norm{\m-c}\leq \norm{f(x_0)-c} +\norm{f(x_0)-\m}\leq 3r$.
\end{proof}

Next we state an analogue of the Lebesgue differentiation theorem. Whereas in the Lebesgue differentiation theorem we approximate $f(x)$ by the integral avarages $\frac{1}{\measure{Q}}\int_{Q}f$, in this analogous theorem we approximate $f(x)$ by $\m$. The theorem was proven by Fujii for medians \cite{fujii1991}.

\begin{lemma}[Fujii's Lemma]\label{fuji}Let $\lambda\in(0,\oh)$. Suppose that $f:\rn\to E$ is a strongly measurable function. Then $\lim_{\substack{Q \text{ a cube,}\\Q\ni x,  \abs{Q}\to 0}}\m=f(x)$ for almost every $x\in\rn$.\end{lemma}
\begin{proof}
We want to show that there exists a null set $N$ such that for all $\epsilon>0$ and for all $x\in\rn\setminus N$ there exists $\delta>0$ such that $\norm{f(x)-\m}<4\epsilon$ for all $\m$ whenever $Q$ is a cube, $Q\ni x$ and $\abs{Q}<\delta$.

Recall that, by definition, $f$ is strongly measurable, if and only if $f$ is both essentially separably valued and Lebesgue-measurable. There exists a null set $N_0$ and a dense countable subset $\{c_i\}_{i=1}^\infty$ of $f(\rn\setminus N_0)$, because $f$ is essentially separably valued. Define the sets $$E^k_j:=\{\norm{f-c_j}\leq\frac{1}{k}\}$$ for all $k,j\in\mathbb{N}_+$. For any fixed $k\in\mathbb{N}_+$ the sets $\{E^k_j\}_{j\in\mathbb{N}_+}$ cover $\rn\setminus N_0$ because $\{c_i\}_{i=1}^\infty$ is dense in $f(\rn\setminus N_0)$. The sets $\{E^k_j\}_{j,k\in\mathbb{N}_+}$  are Lebesgue-measurable, because $f$ is Lebesgue-measurable. 

By the Lebesgue Density Theorem for all $j,k\in\mathbb{N_+}$ there exists a null set $N^k_j\subset E^k_j$ such that for all $x\in E^k_j\setminus N^k_j$ there exists $\delta>0$ such that $\abs{Q\cap E^k_j}\geq(1-\lambda)\abs{Q}$ whenever $Q$ is a cube, $Q\ni x$ and $ \abs{Q}<\delta$.

Define the null set $N:=N_0\cup\bigcup_{j,k\in\mathbb{N}_+}N^k_j$.  Let $\epsilon>0$ and $x\in\rn\setminus N$. Fix a $k\in\mathbb{N_+}$ so large that $\frac{1}{k}<\epsilon$. Now there exists $j$ such that $x\in E^{k}_{j}$ and $\delta>0$ such that \begin{equation}\label{fujilemmavalivaihe}\abs{Q\cap\{\norm{f-c_{j}}\leq\frac{1}{k}\}}=\abs{Q\cap E^{k}_{j}}\geq(1-\lambda)\abs{Q}\end{equation} whenever $Q$ is a cube, $Q\ni x$ and $ \abs{Q}<\delta$. Let $Q$ be such a cube. Let $\m$ be a quasi-optimal center of oscillation. By Lemma \ref{3rproperty} together with the inequality \eqref{fujilemmavalivaihe} we have that $\norm{c_\lambda(f;Q)-c_{j}}\leq3\frac{1}{k}$. This together with the triangle inequality implies that $\norm{c_\lambda(f;Q)-f(x)}\leq \norm{c_\lambda(f;Q)-c_{j}}+ \norm{f(x)-c_{j}}\leq \frac{4}{k}< 4\epsilon.$
\end{proof}

Recall that a quasi-optimal center is defined as any vector $c_\lambda(f;Q)$ that satisfies $\thesymbolforalpha_\lambda(f-c_\lambda(f;Q);Q)\leq 2 \omega_\lambda(f;Q)$.  Roughly speaking, the following lemma says that a quasi-optimal center of oscillation over a cube with a large (but less than half) portion disregarded is also a quasi-optimal center of oscillation over the cube with a smaller portion disregarded.

\begin{lem}\label{keylemma2}Suppose that $0<\lambda\leq\kappa<\oh$. Let $c\in E$. Then for all quasi-optimal centers of oscillation $c_\kappa(f;Q)$ we have
\be 
\thesymbolforalpha_\lambda(f-c_\kappa(f;Q);Q)\leq 4\, \omega_\lambda(f;Q).
\ee
\end{lem} 

\begin{proof}
Let $0<\lambda\leq\kappa<\oh$. Suppose that $c\in E$. Let $c_\kappa(f;Q)$ be a quasi-optimal center of oscillation. Since the function $f$ and the cube $Q$ are fixed for this proof, we suppress "$f$" and "$Q$" in the notation by denoting $c_\kappa:=c_\kappa(f;Q)$, $c_\lambda:=c_\lambda(f;Q)$, $\thesymbolforalpha_\lambda(c_\kappa):= \thesymbolforalpha_\lambda(f-c_\kappa(f;Q);Q)$,  $\thesymbolforalpha_\lambda(c):= \thesymbolforalpha_\lambda(f-c;Q)$, and $\omega_\lambda:=\omega_\lambda(f;Q)$. By the definition of $\thesymbolforalpha_\lambda(c)$ and the assumption $\lambda\leq\kappa$ we have 
\be \measure{Q\cap\{\norm{f-c}>\thesymbolforalpha_\lambda(c)\}}\leq\lambda\abs{Q}\leq\kappa\measure{Q}.
\ee
Therefore Lemma \ref{3rproperty} together with the assumption $\kappa<\oh$ implies that 
\be
\norm{c-c_\kappa}\leq 3 \thesymbolforalpha_\lambda(c).
\ee
Now let $x\in\rn$. If $\norm{f(x)-c}\leq \thesymbolforalpha_\lambda(c)$, then $\norm{f(x)-c_\kappa}\leq \norm{f(x)-c}+\norm{c-c_\kappa(f;Q)}\leq \thesymbolforalpha_\lambda(c)+3\thesymbolforalpha_\lambda(c)= 4\thesymbolforalpha_\lambda(c)$. By contrapositive, if $\norm{f(x)-c_\kappa}>4\thesymbolforalpha_\lambda(c)$, then $\norm{f(x)-c}> \thesymbolforalpha_\lambda(c)$. This together with the definition of $\thesymbolforalpha_\lambda(c)$ implies that
\be
\measure{Q\cap\{\norm{f-c_\kappa}>4\thesymbolforalpha_\lambda(c)\}}\leq \measure{Q\cap\{\norm{f-c}>\thesymbolforalpha_\lambda(c)\}}\leq\lambda\measure{Q}.
\ee
Recall that by definition $\thesymbolforalpha_\lambda(c_\kappa)=\min\{\thesymbolforalpha\geq0: \abs{Q\cap\{\norm{f-c_\kappa}>\thesymbolforalpha\}} \leq \lambda\abs{Q}\}$. Hence 
\begin{equation}
\label{cgene1}
\thesymbolforalpha_\lambda(c_\kappa)\leq 4 \thesymbolforalpha_\lambda(c).
\end{equation}
Therefore by the definition of the greatest lower bound and by the definition $\omega_\lambda=\inf_{c\in E}\rho_\lambda(c)$ we obtain
$$
\thesymbolforalpha_\lambda(c_\kappa)\leq 4 \inf_{c\in E} \rho_\lambda(c)=4\,\omega_\lambda.
$$

\end{proof}

Now we are in position to prove a key lemma, from which the vector-valued generalization of Lerner's local oscillation decomposition formula follows by iterating.

\begin{lem}
\label{keylemma}
Suppose that $f:\rn\to E$ is a strongly measurable function. Let $0<\lambda\leq\kappa<\oh$. Suppose that $Q^0$ is a cube. Then there exists a (possibly empty) collection $\{Q^1_j\}$ of pairwise disjoint dyadic subcubes of $Q^0$ such that for all quasi-optimal centers of oscillation $c_\kappa(f;Q^0)$ and $c_\kappa(f;Q^1_j)$ and for almost every $x\in  Q^0\setminus\bigcup Q^1_j$ we have
\begin{subequations}\label{grp}
\begin{align}
\norm{c_\kappa(f;Q^0)-c_\kappa(f;Q^1_j)}&\leq 3\thesymbolforalpha_\lambda(f-c_\kappa(f;Q^0);Q^0),
\label{le1}\\
 \norm{f(x)-c_\kappa(f;Q^0)}&\leq 3\thesymbolforalpha_\lambda(f-c_\kappa(f;Q^0);Q^0),
 \label{le2}
\end{align}
\end{subequations}
and 
\begin{equation}
\label{le3}
\sum\abs{Q^1_j}\leq 2^d\frac{\lambda}{\kappa}\abs{Q^0}.
\end{equation}
Furthermore, we have the pointwise decomposition 

\begin{eqnarray}
\label{le4}
(f-c_{\kappa}(f;Q^0))1_{Q^0}&=&(f-c_{\kappa}(f;Q^0))1_{Q^0\setminus\bigcup_{j}Q^{1}_j}+\sum_{j}(c_{\kappa}(f;Q^{1}_j)-c_{\kappa}(f;Q^{0}))1_{Q^{1}_j}\nonumber\\
&&+\sum_{j}(f-c_{\kappa}(f;Q^{1}_j)1_{Q^{1}_j},
\end{eqnarray}
which, by the estimates \eqref{le1} and \eqref{le2}, yields almost everywhere the pointwise domination

\begin{equation}
\label{domination}
\norm{f-c_{\kappa}(f;Q^0)}1_{Q^0}\leq 3 \thesymbolforalpha_\lambda(f-c_\kappa(f;Q^0);Q^0)1_{Q^0}+\sum_{j} \norm{f-c_{\kappa}(f;Q^1_j)}1_{Q^1_j}.
\end{equation}
\end{lem}

 \begin{proof}
 
Since the function $f$ is fixed and since the cubes  in the proof are indexed, we suppress "f" in the notation and we refer to the cubes by their indices by denoting $c^0_\kappa:=c_\kappa(f;Q^0)$, $c^1_{\kappa\,j}:=c_\kappa(f;Q^1_j)$, $\thesymbolforalpha^0_\lambda(c^0_\kappa):=\thesymbolforalpha_\lambda(f-c_\kappa(f;Q^0);Q^0)$, and $\omega^0_\lambda:=\omega_\lambda(f;Q^0)$. 

The idea behind the lemma is as follows. We want to decompose the function iteratively with respect to collections of subcubes using the notion of a quasi-optimal center of oscillation $c_\kappa(f;Q)$ and the notion of the least oscillatory bound $\thesymbolforalpha_\lambda(f-c_\kappa(f;Q);Q)$ about a quasi-optimal center of oscillation. 
Let $\{Q^1_j\}$ be a collection of pairwise disjoint dyadic subcubes of $Q^0$. By adding and subtracting we can decompose the function as 
\begin{equation}
\label{decompositioneq}
\underbrace{(f-c_{\kappa}^0)1_{Q^0}}_{T_1}=\underbrace{(f-c_{\kappa}^0)1_{Q^0\setminus\bigcup_{j}Q^{1}_j}}_{=:T_2}+\underbrace{\sum_{j}(c_{\kappa\,j}^1-c_{\kappa}^0)1_{Q^{1}_j}}_{=:T_3}+\underbrace{\sum_{j}(f-c_{\kappa\,j}^1)1_{Q^{1}_j}}_{=:T_4}.
\end{equation}

Observe that the term $T_4$ has the same form as the term $T_1$. Hence we can iterate the equation. We want to control the norm of the terms $T_2$ and $T_3$ by the optimal oscillatory bound $\omega_\lambda(f;Q^0)$ and we want to control the measure of the support of the term $T_4$, in order to make the iteration converge. The question is how to choose the collection. 

By definition, the estimate $\norm{f-c_\kappa^0}\leq\thesymbolforalpha^0_\lambda(c_\kappa^0)$ fails on at most a $\lambda$-portion of $Q^0$. Hence it is natural to consider dyadic subcubes of $Q^0$ such that the estimate fails on more than a $\kappa$-portion. This in mind we define $\{Q^1_j\}$ to be the (possibly empty) collection of the maximal (with respect to the set inclusion) dyadic subcubes $Q^1_j$ of $Q^0$  such that for at least one child $Q^1_{j\,(\text{child})}$ of each $Q^1_j$ we have
\begin{equation}
\label{maximality2}
\measure{Q^1_{j\,(\text{child})}\cap\{\norm{f-c^0_\kappa }>\thesymbolforalpha^0_\lambda(c^0_\kappa)\}}>\kappa\measure{Q^1_{j\,(\text{child})}}.
\end{equation}
By maximality and the nestedness of dyadic cubes the collection $\{Q^1_j\}$ is pairwise disjoint. 

Next we consider the norm estimates \eqref{grp}.  First we check the inequality \eqref{le1}. Consider $Q^1_j$. By maximality $Q^1_j$ itself does not satisfy the inequality \eqref{maximality2}. Therefore it satisfies the opposite inequality
\be
\abs{Q^1_j\cap\{\norm{f-c^0_\kappa }>\thesymbolforalpha^0_\lambda(c^0_\kappa)\}}\leq\kappa\abs{Q^1_j}.
\ee 
Hence, by Lemma \ref{3rproperty}, for each $c^1_{\kappa\,j}$ we have
\be 
\norm{c^0_\kappa-c^1_{\kappa\,j}}\leq 3 \thesymbolforalpha^0_\lambda(c^0_\kappa).
\ee
Next we check the inequality \eqref{le2}. Consider $x\in  Q^0\setminus\bigcup Q^1_j$. Let $Q$ be a dyadic subcube of $Q^0$ containing the point $x$. If $Q$ satisfies the inequality \eqref{maximality2}, then by maximality $Q\subset Q^1_{j}$ for some $Q^1_j$, which implies that $x\in\bigcup Q^1_j$. This is a contradiction. Hence $Q$ satisfies the opposite inequality
\be\measure{Q\cap\{\norm{f-c^0_\kappa }>\thesymbolforalpha^0_\lambda(c^0_\kappa)\}}\leq\kappa\measure{Q}.
\ee
Therefore, by Lemma \ref{3rproperty}, for every dyadic subcube $Q$ of $Q^0$ containing $x$ and for every  $c_\kappa(f;Q)$ 
we have
\be\norm{c^0_\kappa-c_\kappa(f,Q)}\leq 3 \thesymbolforalpha^0_\lambda(c^0_\kappa).
\ee 
 Hence, by Fujii's Lemma, for almost every $x$ we have
\be
\norm{c^0_\kappa-f(x)}=\lim_{\substack{Q \text{ a cube}\\Q\ni x,  \abs{Q}\to 0}}\norm{c^0_\kappa-c_\kappa(f;Q)}\leq 3 \thesymbolforalpha^0_\lambda(c^0_\kappa).
\ee

Next we consider the measure estimate \eqref{le3}. Let $P(x)$ denote the property "$x$ satisfies the inequality $\norm{f(x)-c_\kappa^0}\leq\thesymbolforalpha^0_\lambda(c_\kappa^0)$". By the definition of the collection, for each $Q^1_j$ there is a child $Q^1_{j\,(\text{child})}$ such that 
\begin{equation}
\label{measureinequality1}
\kappa\abs{Q^1_{j\,(\text{child})}}<\measure{Q^1_{j\,(\text{child})}\cap\left\{x\in\rn : P(x) \text{ fails} \right\}}.
\end{equation}
By the definition of $\thesymbolforalpha^0_\lambda(c_\kappa^0)$ we have
\begin{equation}
\label{measureinequality2}
\measure{Q^0\cap\left\{x\in\rn : P(x) \text{ fails} \right\}}\leq\lambda\measure{Q}.
\end{equation}
The inequalities \eqref{measureinequality1} and \eqref{measureinequality2} together with the facts that $Q^1_{j\,(\text{child})}\subset Q^0$, the cubes $Q^1_{j\,(\text{child})}$ are pairwise disjoint, and $\measure{Q^1_j}=2^d \measure{Q^1_{j\,(\text{child})}}$ imply that
\be
\kappa 2^{-d} \sum_j \measure{Q^1_j} \leq \measure{\bigcup Q^1_{j(\text{child})}\cap\left\{x\in\rn : P(x) \text{ fails} \right\}} \leq \measure{Q^0\cap\left\{x\in\rn : P(x) \text{ fails} \right\}}\leq \lambda \measure{Q}.
\ee 

Finally, we observe that the pointwise domination almost everywhere  \eqref{domination} is obtained from the pointwise decomposition \eqref{decompositioneq} by the triangle inequality,

\begin{eqnarray*}
\norm{f-c_{\kappa}^0}1_{Q^0}&\leq& \underbrace{\norm{f-c_{\kappa}^0}}_{\leq 3\thesymbolforalpha^0_\lambda(c^0_\kappa)}1_{Q^0\setminus\bigcup_{j}Q^{1}_j}+\sum_{j}\underbrace{\norm{c_{\kappa\,j}^1-c_{\kappa}^0}}_{\leq 3\thesymbolforalpha^0_\lambda(c^0_\kappa)}1_{Q^{1}_j}+\sum_{j}\norm{f-c_{\kappa\,j}^1}1_{Q^{1}_j}\\
&\leq & 3\thesymbolforalpha^0_\lambda(c^0_\kappa)\underbrace{(1_{Q^0\setminus\bigcup_{j}Q^{1}_j}+\sum_{j}1_{Q^1_j})}_{1_{Q^0}}+\sum_{j}\norm{f-c_{\kappa\,j}^1}1_{Q^{1}_j}.\qedhere
\end{eqnarray*}

\end{proof}

\begin{thm}[Vector-valued generalization of Lerner's local oscillation decomposition formula]\label{lernersformulaend} Suppose that $(E,\norm{\,\cdot\,})$ is a Banach space. Let $\nu\in(0,1)$. 
Suppose that $f:\rn\to E$ is a strongly measurable function. Let $Q^0$ be a cube. 

Then there exists a collection $\cs$ of dyadic subcubes of $Q^0$  such that the collection $\cs$ is pairwise nearly disjoint with the parameter $\nu$ and for almost every $x\in Q^0$ and for every quasi-optimal center of oscillation $c_{1/4}(f;Q^0)$ we have
\be
\norm{f(x)-c_{1/4}(f;Q^0)}\leq 12\sum_{Q\in\cs}\omega_{(1-\nu)2^{-d-2}}(f;Q)1_{Q}(x).
\ee
\end{thm}
\begin{rem} The collection $\cs$ in the statement of Theorem \ref{lernersformulaend} is in fact pairwise nearly disjoint with $E(Q):=Q\setminus\bigcup_{Q' : Q'\subsetneq Q} Q'$ as the set $E(Q)$ in the definition of a pairwise nearly disjoint collection, Definition \ref{nearlydisjoint}. 
\end{rem}
\begin{proof} The theorem is obtained by iterating Lemma \ref{keylemma}. Lerner's Formula for real-valued measurable functions \cite{lerner2010,arXiv:1202.1860} is obtained by using a similar approach. However, whereas Lerner first iterates the equality \eqref{le4} and then estimates the resulting equality, we first estimate the equality \eqref{le4} 
and then iterate the resulting inequality \eqref{domination}. As observed in \cite{hytonen2012}, in this way we avoid introducing the dyadic maximal operator
$$
M^{\#,d}_{\lambda;Q}f(x):=\sup_{\substack{ Q'\subset Q,\\Q' \text{ dyadic cube}}} 1_{Q'}(x)\omega_\lambda(f;Q').
$$

In what follows we work through the details of the iteration in the vector-valued case.  Since the function $f$ is fixed for the proof, we suppress $"f"$ in the notation by denoting $c_\kappa(Q):=c_\kappa(f;Q)$ and $\omega_\lambda(Q):=\omega_\lambda(f;Q)$. For each $k\in\mathbb{N}$ we define recursively a collection $\cs^k$ of dyadic cubes as follows:
\begin{enumerate}
\item[(R1)] Let $\cs^0=\{Q^0\}$. 
\item[(R2)] By using Lemma \ref{keylemma2} and Lemma \ref{keylemma} with the parameters  $\lambda=2^{-d-2}(1-\nu)$ and $\kappa=2^{-2}$ we have that for every cube $Q\in \cs^k$ there exists a collection $\cs^{k+1}(Q)$ of dyadic subcubes of $Q$ and a null set $N^{k+1}(Q)\subset Q$ such that the cubes $Q' \in\cs^{k+1}(Q)$ are pairwise disjoint, 
$$
\sum_{Q' \in\cs^{k+1}(Q)}  \measure{Q'}\leq (1-\nu)\measure{Q},
$$
and for every $x\in Q\setminus N^{k+1}(Q)$
$$
\norm{f-c_{\kappa}(Q)}1_{Q}(x)\leq 12\, \omega_\lambda(Q) 1_{Q}(x)+\sum_{Q' \in\cs^{k+1}(Q)}\norm{f-c_{\kappa}(Q')}1_{Q'}(x).
$$ Define the collection $\cs^{k+1}$ by setting $\cs^{k+1}:=\bigcup_{Q\in\cs^k} \cs^{k+1}(Q)$. Define the null set $N^{k+1}$ by setting $N^0:=\emptyset $ and $N^{k+1}:=\bigcup_{Q\in\cs^k} N^{k+1}(Q)$ for $k\geq0$.
\end{enumerate}
We make the following observations, which follow by induction in $K$ from the recursive definition.
\begin{enumerate}
\item[(O1)]For each $K\in\mathbb{N}$ and for every $x\in Q^0\setminus \bigcup_{k=0}^K N^{k}$ we have 
\begin{equation}
\label{kdomination2}
\norm{f-c_{\kappa}(Q^0)}1_{Q^0}(x)\leq 12\, \sum_{Q\in \bigcup_{k=0}^K\cs^k}\omega_\lambda(Q) 1_{Q}(x)+\underbrace{\sum_{Q\in \cs^{K+1}}\norm{f-c_{\kappa}(Q)}1_{Q}(x)}_{=:R_K(x)}.
\end{equation}

Let $\Omega^k:=\bigcup_ {Q\in\cd^k} Q$. 

\item[(O2)] For each $k\in\mathbb{N}$ the cubes $Q\in \cs^k$ are pairwise disjoint.
\item[(O3)] For each $k\in\mathbb{N}$ we have $\Omega^k\supset \Omega^{k+1}.$
\item[(O4)] For each $Q\in\cd^{k}$ we have $\measure{Q\cap \Omega^{k+1}}\leq (1-\nu)\measure{Q}$.
\end{enumerate}

Next we check that the collection $\cs:=\bigcup_{k=0}^\infty \cs^k$ is pairwise nearly disjoint with the parameter $\nu$, in the sense of Definition \ref{nearlydisjoint}. For each $Q\in\cs^k$ we define $E(Q):=Q\setminus \Omega^{k+1}$. By the observation (O2), if $Q'\in\cs^k$ and $Q\in\cs^k$ are such that $Q'\neq Q$, then $Q'\cap Q=\emptyset$. By the observation (O3), if $k'>k$, $Q'\in\cs^{k'}$, and $Q\in\cs^k$, then $Q'\cap E(Q)=\emptyset$. By observation (O4), $\measure{E(Q)}\geq \nu \measure{Q}$. Hence the collection $\cs$ is pairwise nearly disjoint with the parameter $\nu$.

Next we check that the remainder $R_K$ of the $K$th iteration, which is the rightmost term in the inequality \eqref{kdomination2}, vanishes at almost every point whenever $K$ is sufficiently large. Note that because of the observations (O2), (O3), and (O4), we have that
$$
\measure{\Omega^k}\underset{(O3)}{=}\measure{\Omega^k\cap\Omega^{k-1}}\underset{(O2)}{=}\sum_{Q\in\cs^{k-1}}\measure{\Omega^k\cap Q} \underset{(O4)}{\leq} (1-\nu)\sum_{Q\in\cs^{k-1}} \measure{Q}=(1-\nu)\measure{\Omega^{k-1}}.$$
Hence $\measure{\Omega^k}\leq (1-\nu)\measure{\Omega^{k-1}}\leq  (1-\nu)\left((1-\nu)\measure{\Omega^{k-2}}\right)\leq \ldots \leq (1-\nu)^k\measure{Q^0}$. Therefore $$\measure{\bigcap_{k=0}^\infty \Omega^k}=0.$$
Note that the remainder $R_K$  of the $K$th iteration is supported on $\Omega_K$. Let $N:=\left(\bigcup_{k=0}^\infty N^k\right) \cup \left(\bigcap_{k=0}^\infty \Omega^k\right)$. Let $x\in Q^0 \setminus N$. Then $x \notin \Omega^{K_x}$ for some $K_x\in\mathbb{N}$. Since $\Omega^K\supset \Omega^{K+1}$, we have that $x \notin \Omega^{K}$ for every $K\geq K_x$.  Since $R_K$ is supported on $\Omega_K$, we have that $R_K(x)=0$ for every $K\geq K_x.$  Now, by the inequality \eqref{kdomination2}, we have that
$$
\norm{f-c_{\kappa}(Q^0)}1_{Q^0}(x)\leq 12\, \sum_{Q\in \bigcup_{k=0}^\infty\cs^k}\omega_\lambda(Q) 1_{Q}(x).\qedhere
$$
\end{proof}
\bibliographystyle{plain}
\bibliography{references}
\end{document}